\documentclass[11pt, notitlepage]{article}
\usepackage{amssymb,amsmath,comment}
\usepackage{epsfig,enumerate,amsmath,amsfonts,amssymb}
\usepackage{indentfirst}
\usepackage{pstricks}
\usepackage{setspace}
\usepackage{latexsym}
\usepackage{bm}
\usepackage{amsthm}
\usepackage{amsmath}
\usepackage{parskip}
\usepackage{fixmath}
\catcode`\@=11 \@addtoreset{equation}{section}
\def\thesection{\arabic{section}}

\def\theequation{\thesection.\arabic{equation}}
\catcode`\@=12
\usepackage{colortbl}
\usepackage{a4wide}

\newcommand{\ds} {\displaystyle}
\newcommand{\e}{\varepsilon}
\newcommand{\pa} {\partial}
\newcommand{\al} {\alpha}
\newcommand{\ba} {\beta}
\newcommand{\de} {\delta}

\newcommand{\Om} {\Omega}
\newcommand{\ra} {\rightarrow}
\newcommand{\rp} {\rightharpoonup}

\newcommand{\noi} {\noindent}
\newcommand{\na} {\nabla_{\mathbb{H}}}

\newcommand{\mc} {\mathcal}
\newcommand{\mf} {\mathfrak}

\newcommand{\ld} {\langle}
\newcommand{\rd} {\rangle}

\usepackage[all]{xy}
\catcode`\@=11
\def\theequation{\@arabic{\c@section}.\@arabic{\c@equation}}
\catcode`\@=12

\def\R{{I\!\!R}}
\def\H{{I\!\!H}}

\def\QED{\hfill {$\square$}\goodbreak \medskip}
\def\H{\mathbb{H}^N}
\def\HH{\mathbb{H}^{2N}}
\def\R{\mathbb{R}^N}
\def\J{\mathcal{J}_a}
\newtheorem{Theorem}{Theorem}[section]
\newtheorem{Lemma}[Theorem]{Lemma}
\newtheorem{Proposition}[Theorem]{Proposition}

\newtheorem{Remark}[Theorem]{Remark}

\begin{document}
\vspace{0.01in}

\title
{Variational framework and Lewy-Stampacchia type estimates  for nonlocal operators on Heisenberg   group}

\author{ Divya Goel$^{\,1}$\footnote{e-mail: {\tt divyagoel2511@gmail.com, divya.goel@campus.technion.ac.il}}, \ Vicen\c tiu D. R\u adulescu$^{\,{2,3}}$\footnote{e-mail: {\tt vicentiu.radulescu@imar.ro}} \
	and \  K. Sreenadh$^{\,4}$\footnote{
		e-mail: {\tt sreenadh@maths.iitd.ac.in}} \\ \\ $^1\,${\small Department of Mathematics, Technion - Israel Institute of Technology,}\\
	{\small	Haifa, Israel }
	\\ $^2\,${\small Faculty of Applied Mathematics, AGH University of Science and Technology,}\\ {\small al. Mickiewicza 30, 30-059 Krak\'ow, Poland}\\  $^3\,${\small  Department of Mathematics, University of Craiova, 200585 Craiova, Romania}\\  $^4\,${\small Department of Mathematics, Indian Institute of Technology Delhi,}\\
{\small	Hauz Khaz, New Delhi-110016, India }}

\date{}

\maketitle

\begin{abstract}
   The aim of this article is to derive some Lewy-Stampacchia estimates and existence of solutions for equations driven by a nonlocal integrodifferential operator on Heisenberg group.
\medskip

\noi \textbf{Key words}: Variational inequalities,  Integrodifferential operators, Heisenberg group, Mountain Pass theorem

\medskip

\noi \textit{2010 Mathematics Subject Classification:  35R03, 35R11, 49J40, 35H20.}

\end{abstract}
\newpage
\section{Introduction}
 The Heisenberg Group $\H = \mathbb{R}^N\times\mathbb{R}^N\times \mathbb{R},\; N\in \mathbb{N}$ is a  Lie group, endowed with the  following group law
\begin{align*}
	(x,y,t)\cdot (x^\prime,y^\prime,t^\prime	)= (x+x^\prime,y+y^\prime, t+t^\prime+2( \ld x^\prime, y \rd - \ld x,y^\prime\rd  ))
\end{align*}
where $x,y, x^\prime, y^\prime \in \mathbb{R}^N$. The corresponding Lie algebra of left invariant vector fields  is generated by  the following vector  fields
$$
X_j= \frac{\pa}{\pa x_j}+ 2y_j\frac{\pa}{\pa t},  \;
Y_j= \frac{\pa}{\pa y_j}- 2x_j\frac{\pa}{\pa t} , \;
T= \frac{\pa}{\pa t}.$$
It is straightforward to check that for all $ j,k =1,2,\cdots, N,$
\begin{equation}\label{hg33}
	[X_j,X_k]=[Y_j,Y_k]= [X_j, \frac{\pa}{\pa t}]= [Y_j, \frac{\pa}{\pa t}]=0,
	\text{ and } [X_j, Y_k]= -4\de_{jk}\frac{\pa}{\pa t}.
\end{equation}
These relations \eqref{hg33} establish the Heisenberg's canonical commutation relations of quantum mechanics for position and momentum, hence the name Heisenberg group \cite{hei}.

 We define the left translations on $\H$ by
\begin{align*}
	\tau_\xi:\H \ra \H \qquad \tau_\xi(\xi^\prime)= \xi \cdot\xi^\prime
\end{align*}
and the  natural $\mathbb{H}$-dilations $\de_\theta: \H \ra \H $ by
$$
\de_\theta(x,y,t)= (\theta x, \theta y, \theta^2 t)
$$
for $\theta>0$.
The Jacobian determinant of $\de_\theta$ is $\theta^Q$. The number $Q=2N+2$ is called the homogeneous dimension of $\H$ and it portrays a role equivalent to the topological dimension in the Euclidean space.  We denote the homogeneous norm on $\H$ by
\begin{align*}
	|\xi|= |(x,y,t)|= (t^2+ (x^2+y^2)^2)^{1/4}, \text{ for all } \xi = (x,y,t) \in \H.
\end{align*}

We shall denote $B_r(\xi)$, the ball of center $\xi$ and radius $r$. It implies $\tau_\xi(B_r(0))= B_r(\xi)$ and $\de_r(B_1(0))= B_r(0)$.

Recently, many researchers explored the nonlocal operators on Heisenberg group.   In \cite{frank}, Frank, Gonaz\'alez, Monticelli, and Tan showed that the conformally invariant fractional powers of the sub-Laplacian on the Heisenberg group are given in terms of the scattering operator for an extension problem to the Siegel upper halfspace. Remarkably, this extension problem is different from the one studied, among others, by Caffarelli and Silvestre \cite{caffe}. In \cite{guidi}, Guidi, Maalaoui, and Martino studied the Palais-Smale sequence of the conformally invariant fractional powers of the sub-Laplacian  and proved the existence of solutions.  In \cite{liu},Liu, Wang, and Xiao discussed the nonegative solutions  of a fractional sub-Laplacian
differential inequality on Heisenberg group. In \cite{ele}, Cinti and Tan  established a Liouville-type theorem for a subcritical nonlinear problem, involving a fractional power of the sub-Laplacian in the Heisenberg group. To prove their result authors used the local realization of fractional CR covariant operators, which can be constructed as the Dirichlet-to-Neumann operator of a degenerate elliptic equation  as established in \cite{frank}. 
Nonlocal equations with Convolution type nonlinearities had been discussed  by Goel and Sreenadh \cite{DS}. Their authors established the Brezis-Norenberh type result for the critical problem.  But there is no article which deals a general integro-differntail operator over $\H$. In this article we consider the following integro-differential operator
\begin{align*}
	\mf L_{\mc K} u(\xi)= \frac{1}{2} \int_{ \H} (u(\xi\eta)+u(\xi\eta^{-1})-2u(\xi))\mc K(\eta)~d\eta
\end{align*}
where  $\mc K: \H\setminus\{0\} \ra (0,+\infty)$ be a function with following properties
\begin{equation}\label{hf18}
	\begin{aligned}
		& \theta\mc K \in L^1(\H) \text{ where } \theta(\xi)= \min\{1,|\xi|^2 \}\\
		\text{  there exists }& \mu >0 \text{  such that for all   } \xi \in \H \setminus\{ 0\}, ~\mc K(\xi)\geq \mu |\xi|^{Q+2s},~Q>2s\\
		& \mc K(\xi)=\mc K(\xi^{-1}) \text{ for all } \xi \in \H\setminus\{0\}.
	\end{aligned}
\end{equation}
Employing \eqref{hf18}, one can easily prove that
\begin{align}\label{hf40}
	\mf L_{\mc K} u(\xi)=  \int_{\H} (u(\eta)-u(\xi))\mc K(\eta^{-1}\xi)~d\eta
\end{align}
  In \cite{thang}, Roncal and Thangavelu proved the  \eqref{hf40}  with $\mc K= |\xi|^{Q+2s}$  is the integral  representation  of  fractional sub-Laplacian  on the Heisenberg group.


In case of $\R$, $\mf L_{\mc K}$ is defined as
\begin{align*}
	\mf L_{\{\mc K,\R\}} u(x)= \frac{1}{2} \int_{ \R} (u(x+y)+u(x-y)-2u(x))\mc K(y)~dy
\end{align*}
where  $\mc K: \R\setminus\{0\} \ra (0,+\infty)$ be a function with following properties
  \begin{equation*}
  	\begin{aligned}
  		& \theta\mc K \in L^1(\R) \text{ where } \theta(x)= \min\{1,|x|^2 \}\\
  		\text{  there exists }& \mu >0 \text{  such that for all   } x \in \R \setminus\{ 0\}, ~\mc K(x)\geq \mu |x|^{N+2s}\\
  		& \mc K(x)=\mc K(-x) \text{ for all } x \in \R\setminus\{0\}.
  	\end{aligned}
  \end{equation*}
In recent decade, the subject of  nonlocal elliptic equations involving $\mf L_{\{\mc K,\R\}}$
has gained more popularity because of many  applications such as continuum mechanics,  game theory and phase transition phenomena.
For an extensive survey on integro-differential operators  and their applications, one may refer to \cite{valdi,  servadei,stinga} and references therein.

To proceed further we defined the following space
\begin{align*}
	\mc Z= \{ u :\H \ra \mathbb{R}~: u \in L^2(\Om), (u(\xi )-u(\eta))\sqrt{\mc K(\eta^{-1}\xi )} \in   L^2(\mc S,d\xi d \eta )  \}
\end{align*}
with $ \mc S= \H \setminus(\Om^c \times \Om^c)$.

In this article we have defined an integro-differential operator on $\H$ and investigate the Lewy-Stampacchia  estimates. Moreover, we proved the existence of solution of a subcritical  problem involving the operator $\mf{L_\mc K}$ by establishing the compact embedding of the space $\mc Z_0$ (see Section 2).   In this regard, the results proved in the present article are completely new.
    The main results proved in this article are the following.
\begin{Theorem}\label{thmhf1}
Let $\Om$ be an bounded extension domain  in $\H$ and $f \in L^\infty(\Om)$.
Assume $u_0 \in \mc Z \cap L^\infty(\H\setminus\Om)$, $\phi \in \mc Z$ with $u_0\leq \phi $ a.e in $\H$ and $\mf L_{\mc K} \phi \in L^\infty(\Om)$. If  $\mc M_\phi= \{u \in \mc Z~:~ u =u_0 \text{ in } \Om^c,~ u \leq \phi \text{ in } \Om \}$ and $u \in \mc M_\phi$ is a solution of the variational inequality
\begin{equation}\label{hf28}
			\ds \int_{\mc S} (u(\xi)-u(\eta))((v-u)(\xi)-(v-u)(\eta)) \mc K(\eta^{-1}\xi )d \xi  d \eta
			\geq \ds \int_{\Om}  f(v-u) d\xi,~ \text{ for all } v\in \mc M_\phi
\end{equation}
then
\begin{equation}\label{hf29}
\begin{aligned}
0 & \leq - \int_{\mc S} (u(\xi)-u(\eta))(\psi(\xi)-\psi(\eta)) \mc K(\eta^{-1}\xi )d \xi  d \eta + \int_{\Om} f \psi~d\xi \leq \int_{\Om}  (\mf L_{\mc K} \phi +f )^+ \psi ~d \xi
\end{aligned}
\end{equation}
for all non-negative functions $\psi\in C^\infty_c(\Om).$

\end{Theorem}
\begin{Theorem}\label{thmhf2}
Let $f$ be a Carath\'eodory function satisfying the following conditions
\begin{equation}\label{hf30}
\begin{aligned}
 \left\{
\begin{array}{ll}
& \text{there exists } a_1,a_2>0 \text{ and } q \in (2,Q^*),Q^*=\frac{2Q}{Q-2s} \text{  such that  }\\
& |f(\xi,l)|\leq a_1+a_2|l|^{q-1} \text{ a.e } \xi \in \Om,~l\in \mathbb{R};\\
&  \ds \lim_{|l|\ra 0} \frac{f(\xi,l)}{|l|}=0\text{ uniformly in  }  x \in \Om;\\
& \text{there exist } \vartheta >2 \text{ and } \mc R>0 \text{ such that  a.e } x \in \Om, ~l \in \mathbb{R}, |l|>\mc R,\\ &\quad \quad 0<\vartheta F(\xi,l)\leq lf(\xi,l)\\
& \text{ where } F \text{ is the primotive of } f .
\end{array}
\right.
\end{aligned}
\end{equation}
Then the following problem

\begin{equation*}
\begin{aligned}
\mf{(P)}\left\{
\begin{array}{ll}
\mf L_{\mc K} u & = f(\xi,u) \text{ in }  \Om, u  = 0    \text{ in } \H\setminus \Om
\end{array}
\right.
\end{aligned}
\end{equation*}
\end{Theorem}
has a mountain pass type solution which is not identically zero.

Turing to layout of the paper, in Section 2, we will give the  variational framework, fibering
map analysis and compactness of Palais-Smale sequences. In section 3, we gave the proof of Theorem \ref{thmhf1}. In section 4, we gave the proof of Theorem \ref{thmhf2}.
\section{Preliminaries}
In this section, we state some known results rquired for the variational framework.
Let $\mc M$ is the linear
 subspace of Lebesgue measurable functions from $\H$ to $\mathbb{R}$  with the following property
\begin{align*}
\text{ if } u \in \mc M \text{ then } u|_\Om \in L^1(\Om).
\end{align*}
Let $A,B \subset \mc M $ be  the sets such that the following product is well defined
\begin{align*}
P: A\times B\ra L^1(\mc S, d\xi,d \eta) \qquad (a,b) \mapsto P(a,b):=ab.
\end{align*}
Let $\mc W$ be the set containing all the non negative constants and the function $u_0-\phi$, where functions $u_0, \phi \in \mc M$ such that $u_0\leq \phi$ a.e in $\H$. Further, we consider a linear subspaces $\tilde{\mc N_0},~ \mc N_0$  such that $\tilde{\mc N_0} \subset \mc N_0\subset \mc M$ and $\mc N_0$ satisfying the following property
\begin{align}\label{hf3}
\text{ if } v \in \mc N_0 \text{ and } w \in \mc W, \text{ then } (v+w)^+ \in \mc N_0
\end{align}
Let $g: \mc M\ra A$ and $h: \mc M\ra B$ be two well-defined operators and let   $\mc J :\mc M\times \mc M \ra \mathbb{R}$ be the functional defined by
\begin{align*}
\mf J (\varphi, \psi )= -\int_{\mc S} g(\varphi)h(\psi )d\xi d \eta.
\end{align*}
 Throughout the article we assume the following assumptions on the functional $g,h$ and $\mf J$:
\begin{align}\label{hf4}
\left\{
\begin{array}{ll}
g(u+r)= g(u), ~h(-u)= -h(u) \text{ for all } u \in \mc M \text{ and } r \in \mathbb{R}; \\
\text{ if } u, v \in \mc M \text{ such that } (u-v)^+ \in \mc N_0 \text{ and } \mf J(u,(u-v)^+)\geq \mf J(v,(u-v)^+)\\
 \text{ then } u \leq v \text{ a.e  in } \Om.
\end{array}
\right.
\end{align}
 For $u \in \mc M,$ if there exists $\Upsilon_u\in L^\infty(\Om)$  such that $\mf J(u,v)= \int_{\Om} \Upsilon_u v ~dx$ for all $v \in \mc N_0$ then we denote $\mf J(u):= \Upsilon_u$ and $ \mf J(u) \in L^\infty(\Om)$.
	Define the cut-off function
	\begin{equation}\label{hf16}
		\begin{aligned}
			D_r(l) =  \left\{
			\begin{array}{ll}
				0  &  \text{ if } l\leq 0, \\
				l/r  & \text{ if }  0<l<r,\\
				1 & \text{ if } l\geq r
			\end{array}
			\right.
		\end{aligned}
	\end{equation}
	where  $r \in (0,1)$.
	Suppose  $\mf J(\phi) \in L^\infty(\Om)$ and $f\in L^\infty (\Om)$. We assume that there exists $u_r \in \mc M$ such that $ u_r-u_0\in \mc N_0$  and satisfies the following
	\begin{align*}
			\mf{J}(u_r, \varphi) = \ds\int_{\Om} [( (\mf J(\phi) +f)^+)(1-D_r(\phi-u_r))-f]\varphi  d \xi \text{ for all  } \varphi \in \mc N_0.
	\end{align*}

 Also, if $u_r \rightarrow u$ uniformly in $\mathbb R^N$ as $r\rightarrow 0$ then up to a subsequence,
$\mf J(u_r,\varphi) \ra \mf J(u,\varphi) \text{ for all } \varphi \in  \tilde{\mc N_0}.$ Then we have the following theorem from \cite{ser}:
\begin{Theorem}\label{thmhf3}
	Let $u_0, \phi$ and  $\mf J $ satisfy the above assumptions and if  $u \in \mc M$ is such that $u-u_0 \in \mc N_0,$ $u \leq \phi  \text{  a.e  in  } \Om$ and is a  solution of the following  variational inequality
	\begin{align*}
		\left\{
		\begin{array}{ll}
			  \ds \int_{\mc S} g(u)h(v-u)d \mu \geq \ds \int_{\Om}  f(v-u) d\xi\\ \text{ for all } v\in \{ \mc M \; :\; v-u_0 \in \mc N_0\}.
		\end{array}
		\right.
	\end{align*}
	 then
	\begin{equation*}
	0\leq \mf J(u,\varphi)+ \int_{\Om} f \varphi~ d\xi \leq \int_{\Om} (\mf J(\phi)+f)^+ \varphi~d\xi
	\end{equation*}
	for all $\varphi \in \tilde{\mc N_0},~ \varphi \geq 0$ a.e. in $\Om$.
\end{Theorem}
\subsection{Variational Framework for $\mf L_{\mc K}$}

In this section we give the variational setup for the operator $\mf L_{\mc K}$. Define the following norm on the space $\mc Z$
 \begin{align*}
 \|u\|_{\mc Z} = \|u\|_{L^2(\Om)}+ \left( \int_{\mc S} |u(\xi)-u(\eta)|^2\mc K(\eta^{-1}\xi)~d\xi d \eta   \right)^{1/2}
 \end{align*}

\begin{Lemma}
	Let $\phi \in C^2_0(\Om)$. Then $  |\phi(\xi)-\phi(\eta)|^2 \mc K(\eta^{-1}\xi) \in  L^1(\mathbb{H}^{2N})$.
\end{Lemma}
\begin{proof}
	Since $\phi =0 $ on $\H \setminus \Om$,
	\begin{equation}\label{fr1}
\begin{aligned}
	\int_{\HH} |\phi(\xi)-\phi(\eta)|^2\mc K(\eta^{-1}\xi) ~d\xi d\eta & = 	\int_{\mc S} |\phi(\xi)-\phi(\eta)|^2\mc K(\eta^{-1}\xi) ~d\xi d\eta \\
	& \leq  2
		\int_{\Om \times \H} |\phi(\xi)-\phi(\eta)|^2\mc K(\eta^{-1}\xi) ~d\xi d\eta .
	\end{aligned}
		\end{equation}
	The fact that  $\phi \in C^2_0(\Om)$ implies
	\begin{align*}
	|\phi(\xi)-\phi(\eta)|\leq \|\na \phi \|_{L^\infty(\H)} |\xi- \eta | \text{ and } |\phi(\xi)-\phi(\eta)|\leq \| \phi \|_{L^\infty(\H)}.
	\end{align*}
	Hence  using the definition of $\theta$, defined in \eqref{hf18}, we have
	\begin{align}\label{fr2}
	|\phi(\xi)-\phi(\eta)|\leq 2\| \phi \|_{C^1(\H)}\min\{|\xi-\eta|,1  \} = 2 \|\phi\|_{C^1(\H)} \sqrt{\theta(\eta^{-1}\xi)}
	\end{align}
	From \eqref{fr1} and \eqref{fr2} and \eqref{hf18}, we get
	\begin{align*}
	\int_{\HH} |\phi(\xi)-\phi(\eta)|^2\mc K(\eta^{-1}\xi) ~d\xi d\eta & \leq  2
	\int_{\Om \times \H} |\phi(\xi)-\phi(\eta)|^2\mc K(\eta^{-1}\xi) ~d\xi d\eta \\
	& \leq 8 \|\phi\|^2_{C^1(\H)}|\Om|	\int_{ \H} \theta(\xi )\mc K(\xi) ~d\xi d\eta< \infty.
	\end{align*}
	The proof follows. \QED	
\end{proof}

The above Lemma implies $C^2_0(\H) \subset \mc Z$. Now we define the following subspaces
\begin{align*}
& \mc Z_0=\{ u \in \mc Z ~:~ u =0 \text{ a.e  in } \Om^c\},\\
&  \tilde{\mc Z_0} = C_c^\infty(\Om),\\
& A=B= \left\{  u :\HH \ra \mathbb{R}~:~ u|_{\mc S} \in L^2(\mc S, d\xi d\eta) \right\}
\end{align*}
and $P$ is the usual product between functions. Clearly, $\tilde{\mc Z_0} \subseteq \mc Z_0  \subseteq \mc Z$ and $u^+ \in \mc Z$ for all $u \in \mc Z$.

 In \cite{arka}, Mallick and Adimurthi defined the  usual fractional order Sobolev space on $\H$ as follows
\begin{align*}
W^{s,p}(\Om) =\left\{  u: \H \ra \mathbb{R}~:~ u \in L^2(\Om)\text{ and } \int_{\Om\times \Om } \frac{|u(\xi)-u(\eta)|^2}{|\eta^{-1}\xi|^{Q+2s}}~d\xi d \eta <\infty \right\}
\end{align*}
where $ s\in (0,1)$ and $p>1$  as $W^{s,p}_0(\Om)$ is the closure of $C_c^\infty(\Om)$ with respect to the following norm
\begin{align*}
\|u\|_{W^{s,p}(\Om)} = \|u\|_{L^2(\Om)} + \left( \int_{\Om\times \Om } \frac{|u(\xi)-u(\eta)|^2}{|\eta^{-1}\xi|^{Q+2s}}~d\xi d \eta   \right)^{1/2}.
\end{align*}
Clearly, $W^{s,p}_0(\H)= W^{s,p}(\H)$.
\begin{Lemma}\label{lemhf1}
	Let $c(\mu) = \max \{ 1, \mu^{-1/2}  \} $. 	The following assertions holds true:
	\begin{enumerate}
		\item [(i)] If $v \in \mc Z$, then $ v \in W^{s,2}(\Om) $ and $\|v\|_{W^{s,2}(\Om)}   \leq c(\mu) \|v\|_{\mc Z} $.
		\item[(ii)]   if $v \in \mc Z_0$  then $v \in  W^{s,2}(\H)$ and $\|v\|_{W^{s,2}(\Om)}  \leq \|v\|_{W^{s,2}(\R)}  \leq c(\mu) \|v\|_{\mc Z_0} $.
		\item[(iii)]  if $v \in \mc Z_0$  and $Q^*= \frac{2Q}{Q-2s}$, then there exists a positive constant $c(Q,s)$ such that
		\begin{align*}
			\|v\|_{L^{Q^*}(\Om)}^2 \leq c(Q,s ) \int_{\HH } \frac{|v(\xi)-v(\eta)|^2}{|\eta^{-1}\xi|^{Q+2s}}~d\xi d \eta\leq c(Q,s) c^2(\mu) \|v\|_{\mc Z_0}^2  .
		\end{align*}
		\item[(iv)]  The space $\mc Z_0$ is a norm linear space endowed with the norm
		\begin{align}\label{hf14}
		\|u\|_{\mc Z_0} = \left( \int_{\mc S} |u(\xi)-u(\eta)|^2\mc K(\eta^{-1}\xi)~d\xi d \eta   \right)^{1/2}.
		\end{align}Moreover,  there exists a constant $C>0$ such that for any $ v\in \mc Z_0$, we have
		\begin{align}\label{hf13}
		 \|v\|_{\mc Z_0}\leq \|v\|_{\mc Z} \leq C \|v\|_{\mc Z_0}.
		\end{align}
	\end{enumerate}
\end{Lemma}
\begin{proof}
	(i)	Consider
	\begin{align*}
	\int_{\Om\times \Om } \frac{|v(\xi)-v(\eta)|^2}{|\eta^{-1}\xi|^{Q+2s}}~d\xi d \eta & \leq \frac{1}{\mu}  \int_{\Om\times \Om } |v(\xi)-v(\eta)|^2 \mc K(\eta^{-1}\xi)~d\xi d \eta \\
	& \leq \frac{1}{\mu}  \int_{\mc S } |v(\xi)-v(\eta)|^2 \mc K(\eta^{-1}\xi)~d\xi d \eta <\infty.
	\end{align*}
	(ii)		Using the fact that $v=0$ a.e. in $\Om^c$, we get $\|v\|_{L^2(\H)}= \|v\|_{L^2(\Om)}$ and
	\begin{align*}
	\int_{\HH } \frac{|v(\xi)-v(\eta)|^2}{|\eta^{-1}\xi|^{Q+2s}}~d\xi d \eta & =   \int_{\mc S } \frac{|v(\xi)-v(\eta)|^2}{|\eta^{-1}\xi|^{Q+2s}}~d\xi d \eta \\
	& \leq \frac{1}{\mu}  \int_{\mc S } |v(\xi)-v(\eta)|^2 \mc K(\eta^{-1}\xi)~d\xi d \eta <\infty.
	\end{align*}
	(iii) Let $v \in\mc Z_0$ then by (ii), $v \in  W^{s,2}(\H)$. Now using the \cite[Theorem 1.1]{arka}  with $\al=0$, we have
	\begin{align*}
	\|v\|_{L^{Q^*}(\Om)}^2 \leq c \int_{\HH } \frac{|v(\xi)-v(\eta)|^2}{|\eta^{-1}\xi|^{Q+2s}}~d\xi d \eta
	\end{align*}
	where $c$ depends on $Q$ and  $s$.\\
	(iv) 	Let $\|v\|_{\mc Z_0}=0$ then It implies $v(\xi)=v(\eta) $ a.e in $\mc S$. Let $v = c \geq 0 $ a.e. in $ \H$ but $v \in \mc Z_0$  implies $ v=0$ a.e. in $\Om^c$. That is,  $v= 0 $ a.e. in $\H$.  Hence, $\|\cdot \|_{\mc Z_0}$ is a norm. For \eqref{hf13}, by definition of $\|\cdot\|_{\mc Z}$,
	$\|v\|_{\mc Z_0} \leq \|v\|_{\mc Z}$. With the help of H\"older's inequality  and (iii), we get
	\begin{align*}
 \|v\|_{\mc Z}^2 & \leq 2\|v\|_{L^2(\Om)}^2+ 2\int_{\mc S } |v(\xi)-v(\eta)|^2 \mc K(\eta^{-1}\xi)~d\xi d \eta\\
 &\leq  2C_1 \|v\|_{L^{Q^*}(\Om)}^2+ 2\int_{\mc S } |v(\xi)-v(\eta)|^2 \mc K(\eta^{-1}\xi)~d\xi d \eta \\
 & \leq 2cC_1 \int_{\HH } \frac{|v(\xi)-v(\eta)|^2}{|\eta^{-1}\xi|^{Q+2s}}~d\xi d \eta  + 2\int_{\mc S } |v(\xi)-v(\eta)|^2 \mc K(\eta^{-1}\xi)~d\xi d \eta\\
 & \leq 2\left(\frac{cC_1}{\mu}+1\right) \int_{\mc S } |v(\xi)-v(\eta)|^2 \mc K(\eta^{-1}\xi)~d\xi d \eta\\
  & = C_2 \|v\|_{\mc Z_0}^2.
	\end{align*}
	 This proves the desired result  with $C= \sqrt{C_2}$. \QED		
\end{proof}

	\begin{Lemma}
		$\mc Z_0$ is a Hilbert space endowed with the following inner product
		\begin{align*}
		\ld u,v\rd =  \int_{\mc S} (u(\xi)-u(\eta))(v(\xi)-v(\eta)) \mc K(\eta^{-1}\xi )d \xi  d \eta \text{ for all  } u,v \in \mc Z_0
		\end{align*}
	\end{Lemma}
\begin{proof}
	It is easy to prove that $\ld\cdot,\cdot\rd $ is an inner product space and lead to the norm defined in \eqref{hf14}.
	Now we prove that $\mc Z_0$ is a complete with respect to norm $\|\cdot\|_{\mc Z_0}$. Let $u_n$ be a cauchy sequence  in $\mc Z_0$. Hence for any $\e>0$ there exists $n_{\e}$ such that for all $n,m\geq n_{\e}$,
	\begin{align}\label{hf15}
	\|u_n-u_m\|_{L^2(\Om)}^2 \leq \|u_n-u_m\|_{\mc Z}^2 \leq C \|u_n-u_m\|_{\mc Z_0}^2 <\e.
	\end{align}
	Since $L^2(\Om)$ is a complete space, there exists $u_*\in L^2(\Om)$ such that $u_n\ra u_*$ in $L^2(\Om)$ as $n \ra \infty$. Up to a subsequence(denoted by $u_n$)  such that $u_n \ra u_*$ a.e.in $\H$. By using Fatou Lemma, we get
	\begin{align*}
	\|u_*\|^2_{\mc Z_0} & \leq \liminf_{n \ra \infty} \|u_n\|^2_{\mc Z_0} \leq \liminf_{n \ra \infty}\left( \|u_n-u_*\|_{\mc Z_0} +\|u_*\|_{\mc Z_0} \right)^2 < (1+\|u_*\|_{\mc Z_0})^2< \infty
	\end{align*}
	where we used \eqref{hf15} with $\e =1$. It implies $u_*\in \mc Z_0$.  Using \eqref{hf15}, with $m\geq n_{\e}$, we obtain
	\begin{align*}
	\|u_m-u_*\|^2_{\mc Z_0}\leq \|u_m-u_*\|^2_{\mc Z}\leq \liminf_{n \ra \infty} \|u_m-u_n\|^2_{\mc Z} \leq C\liminf_{n \ra \infty} \|u_m-u_n\|^2_{\mc Z_0}  \leq C\e.
	\end{align*}
	It implies that $u_m\ra u_*$ as $m\ra \infty$. Hence, we get the desired result. \QED	
\end{proof}

\begin{Lemma}\label{lemhf2}
	Let $u_n$ is a bounded sequence in $X_0$. Then there exists $u_* \in L^r(\H)$ such that, up to a subsequence, $u_n \ra u_*$ in $L^r(\H)$ as $n \ra \infty$ for all $r \in [1,Q^*)$.
\end{Lemma}

\begin{proof}
	By Lemma \ref{lemhf1} (ii), $u_n \in W^{s,2}(\Om)$ and $u_n$ is a bounded sequence in $W^{s,2}(\Om)$. By using \cite[Theorem 1.4]{arka}, there exists $u_* \in L^r(\Om)$ such that $u_n\ra u_*$ in  $L^r(\Om)$ for any  $r \in [1,Q^*)$. Moreover,  $u_n =0$ a.e. in $\H\setminus\Om$ implies $u_* =0$ a.e. in $\H\setminus\Om$, that is,  $u_n\ra u_*$ in  $L^r(\H)$. \QED
\end{proof}
\section{Proof of Theorem \ref{thmhf1}}
In this section, we gave the proof of Theorem
\ref{thmhf1} using the Theorem \ref{thmhf2}.
For  $u, v \in \mc Z$, we define
$ g(u)(\xi,\eta) = (u(\xi)-u(\eta ))\sqrt{\mc K(\eta^{-1}\xi)} \in A,$ and
$ g(v)(\xi,\eta) = (v(\xi)-v(\eta ))\sqrt{\mc K(\eta^{-1}\xi)} \in A,$
and   $$\mf J(u,v)= \ds \int_{\mc S} (u(\xi)-u(\eta))(v(\xi)-v(\eta)) \mc K(\eta^{-1}\xi )d \xi  d \eta$$
\begin{Lemma}\label{lemhf3}
	The following holds
	\begin{enumerate}
		\item[(i)] $\mf J$ is well defined map.
		\item[(ii)] The assumptions \eqref{hf3} and \eqref{hf4} are satisfied.
	\end{enumerate}
\end{Lemma}
\begin{proof}
	(i) By using the Cauchy-Schwarz inequality and the fact that $u, v \in \mc Z$ , we get
	\begin{align*}
	2 |u(\xi)-u(\eta )||v(\xi)-v(\eta )|\sqrt{\mc K(\eta^{-1}\xi)}\sqrt{\mc K(\eta^{-1}\xi)} \leq \left( |u(\xi)-u(\eta )|^2  + |v(\xi)-v(\eta )|^2 \right)\mc K(\eta^{-1}\xi).
	\end{align*}
	(ii) Let $v \in \mc Z_0$ and $w \in \mc Z$ with $w \leq 0$ a.e. in $\H$. The fact that  $v+w \in \mc Z$ implies $(v+w)^+ \in \mc Z$. Also $v+w(\xi)\leq v(\xi)=0$ a.e. in $\Om^c$, implies $(v+w)^+ =0$ in $\Om^c$. Therefore, $(v+w)^+ \in \mc Z_0$. It proves \eqref{hf3}.\\
	By the definition of $g$ and $h$, one can easily show that
	$g(u+r)= g(u), ~h(-u)= -h(u) \text{ for all } u \in \mc Z \text{ and } r \in \mathbb{R}$.  Let $ u,v \in \mc Z$ with $(u-v)^+ \in \mc Z_0$ and $\mf J(u,(u-v)^+) \geq \mf J(v,(u-v)^+)$. Set $ w = u-v$  and $w= w^+-w^-$. Consider
	\begin{align}\label{hf12}
	(w(\xi)-w(\eta))(w^+(\xi)-w^+(\eta))
	= (w^+(\xi)-w^+(\eta))^2+ w^-(\xi)w^+(\eta)+ w^+(\xi)w^-(\eta)
	\end{align}
	By using \eqref{hf18} and \eqref{hf12}, we have
	\begin{align*}
	0&\geq  \mf J(v,(u-v)^+)- \mf J(u,(u-v)^+)\\
	&   = \ds \int_{\mc S} (w(\xi)-w(\eta))(w^+(\xi)-w^+(\eta)) \mc K(\eta^{-1}\xi )d \xi  d \eta\\
	&   =\ds \int_{\mc S} \left((w^+(\xi)-w^+(\eta))^2+ w^-(\xi)w^+(\eta)+ w^+(\xi)w^-(\eta)\right)  \mc K(\eta^{-1}\xi )d \xi  d \eta \geq 0.
	\end{align*}
	It implies $(w^+(\xi)-w^+(\eta))^2+ w^-(\xi)w^+(\eta)+ w^+(\xi)w^-(\eta)=0$ a.e. in $\mc S$, that is, $w^+(\xi)=w^+(\eta) $ a.e in $\mc S$. Let $w^+ = c \geq 0 $ a.e. in $ \H$ but $w^+ \in \mc Z_0$  implies $ c=0$. Hence $(u-v)^+= 0 $, that is, $u\leq v $ a.e. in $\H$. Thus, \eqref{hf4} is satisfied. \QED
\end{proof}
 We set two functions $u_0 \in \mc Z \cap L^\infty(\Om^c)$ and $\phi \in \mc Z$ with $u_0 \leq \phi$ a.e. in $\H$.
Assume $\mf L_{\mc K} \phi,f \in L^\infty(\Om)$. For a.e $\xi \in\Om$ and $l \in \mathbb{R}$, define  $T= (\mf L_{\mc K} \phi+f)^+ \in L^\infty(\Om)$ and $w_r(\xi,l)= T(\xi)(1-D_r(\phi(\xi)-l))-f(\xi)$   where $D_r$ is defined in \eqref{hf16} and $r \in (0,1)$.

\begin{Proposition}\label{prophf2}
	Let $r \in (0,1)$. Then there exists  $u_r \in \mc Z$ such that is a solution to $\mf L_{\mc K}u_r= w_r(\xi,u_r)$ and  $u_r=u_0 \in \Om^c$. That is,  for all   $v \in \mc Z_0$
	\begin{align}\label{hf17}
	\left\{
	\begin{array}{ll}
	\ds \int_{\mc S} (u_r(\xi)-u_r(\eta))(v(\xi)-v(\eta)) \mc K(\eta^{-1}\xi )d \xi  d \eta  + \ds\int_{\Om} w_r(\xi,u_r(\xi))v  d \xi =0 \\
	u_r\in \mc Z, u_r-u_0\in \mc Z_0.
	\end{array}
	\right.
	\end{align}
\end{Proposition}
\begin{proof}
	Consider the space $\mc Z_{u_0}=\{u \in \mc Z~:~ u-u_0 \in \mc Z_0 \}$ and the functional $I_r:\mc Z_{u_0}\ra \mathbb{R}$ defined as
	\begin{align*}
	I_r(u)= \frac12 \int_{\mc S} |u(\xi)-u(\eta)|^2\mc K(\eta^{-1}\xi)~d\xi d \eta + \ds\int_{\Om} W_r(\xi,u_r(\xi))  d \xi
	\end{align*}
	where $W_r$ is the primitive of $w_r$.  Clearly, by definition of $D_r$,
	 \begin{align*}
		|w_r|& \leq \|T\|_{L^\infty(\Om)}\|(1-D_r(\phi(\xi)-l))\|_{L^\infty(\Om)}+\|f\|_{L^\infty(\Om)}\\
		&\leq  \|T\|_{L^\infty(\Om)}+\|f\|_{L^\infty(\Om)}:= \zeta.
	\end{align*}
	It implies $|W_r(\xi,u(\xi)|\leq \zeta |u(\xi)|$. Employing the Young's inequality, Minkowski inequality, H\"older's inequality, Lemma \ref{lemhf1}(iii) and \eqref{hf18}, we deduce that
	\begin{equation}\label{hf19}
\begin{aligned}
	I_r(u)& \geq \frac12  \int_{\mc S} |u(\xi)-u(\eta)|^2\mc K(\eta^{-1}\xi)~d\xi d \eta -\zeta \int_{\Om} |u(\xi)|~d \xi \\
	& \geq \frac12  \int_{\mc S} |u(\xi)-u(\eta)|^2\mc K(\eta^{-1}\xi)~d\xi d \eta -\frac{\zeta\e}{2} \|u\|^2_{L^2(\Om)}  -\frac{\zeta}{2\e}|\Om|\\
	& \geq \frac{\mu}{2}  \int_{\mc S} \frac{|u(\xi)-u(\eta)|^2}{|\eta^{-1}\xi|^{Q+2s}}\mc ~d\xi d \eta -\zeta\e \|u-u_0\|^2_{L^2(\Om)} -\zeta\e \|u_0\|^2_{L^2(\Om)}  -\frac{\zeta}{2\e}|\Om|\\
	& \geq \mu \|u-u_0\|^2_{\mc Z_0} -\mu \|u_0\|^2_{\mc Z_0} -\zeta\e C|\Om|^{\frac{Q^*-2}{2}}\|u-u_0\|^2_{\mc Z_0} -\zeta\e \|u_0\|^2_{L^2(\Om)}  -\frac{\zeta}{2\e}|\Om|\\
	& = \left(\mu -\zeta\e C|\Om|^{\frac{Q^*-2}{2}}\right) \|u-u_0\|^2_{\mc Z_0}-\mu \|u_0\|^2_{\mc Z_0}  -\zeta\e \|u_0\|^2_{L^2(\Om)} -\frac{\zeta}{2\e}|\Om|.
	\end{aligned}
		\end{equation}
Using the fact that $u_0 \in\mc Z$ and the properties of $\mc K$, we get $\|u_0\|^2{\mc Z_0}, \|u_0\|^2_{L^2(\Om)}<\infty$. Now choosing $\e>0$ such that $\mu -\zeta\e C|\Om|^{\frac{Q^*-2}{2}}>0$, we deduce that
\begin{align*}
I_r(u)& \geq -\mu \|u_0\|^2_{\mc Z_0}  -\zeta\e \|u_0\|^2_{L^2(\Om)} -\frac{\zeta}{2\e}|\Om|>-\infty.
\end{align*}
It implies that  $\ds \inf_{u \in \mc Z_{u_0} }	I_r(u)>-\infty$. Now let $u_n\in \mc Z_{u_0}$ be the minimizing sequence for $I_r$ then $I_r(u_n)\ra \ds \inf_{u \in \mc Z_{u_0} }	I_r(u)$.  Thus from \eqref{hf19} with $u=u_n$ we get $u_n-u_0$ is a bounded sequence is $X_0$. From Lemma \ref{lemhf2}, there exists $u^* \in X_0$ such tha, up to a subsequence,  $u_n-u_0 \ra u^*$ in $L^{\nu}$ for $\nu \in [1,Q^*)$ and $u_n-u_0\ra u^*$ a.e. in $\H$.  Thus $u^* \in \mc Z_0$. Define $u_*= u_0+u^* $. Then $u_* \in \mc Z_{u_0}$. Now using the continuity of the map $i\mapsto W_r(\xi,i)$ for all $\xi \in \H$ and $i \in \mathbb{R}$  and the dominated convergence theorem, we get
\begin{align*}
\lim_{n\ra \infty} I_r(u_n) \geq \frac12 \int_{\mc S} |u_*(\xi)-u_*(\eta)|^2\mc K(\eta^{-1}\xi)~d\xi d \eta+ \int_{\Om} W_r(\xi,u_*(\xi))~d\xi= I_r(u_*).
\end{align*}
Hence $I_r(u_*) =  \ds \inf_{u \in \mc Z_{u_0} }	I_r(u) $. It implies \eqref{hf17} has a solution. \QED
\end{proof}

\begin{Proposition}\label{prophf1}
	Let $v \in C_0^2(\Om)$ and $u \in \mc Z\cap L^\infty(\Om^c)$.  Then
\begin{equation}\label{hf27}
	\begin{aligned}
	 \int_{\mc S} (u(\xi)-u(\eta))& (v(\xi)-v(\eta)) \mc K(\eta^{-1}\xi )d \xi  d \eta\\
	 &  =  \int_{\HH} u(\xi)(2v(\xi)-v(\xi \tilde{\xi})- v(\xi\tilde{\xi}^{-1})) \mc K(\tilde{\xi} )d \xi  d \tilde{\xi}.
	\end{aligned}
	\end{equation}
\end{Proposition}
\begin{proof}
	Let
	\begin{align*}
	& \mf E_{0}= \mc S= \HH\setminus(\Om^c\times \Om^c),\\
	& \mf E_{\de} = \{(\xi,\eta) \in \mf E_0~:~ |\eta^{-1}\xi|\geq \de  \}\\
	& \mf E_{\de}^{+} = \{(\xi,\tilde{\xi}) \in \HH~:~  (\xi,\xi\tilde{\xi}) \in \mf E_0
	\text{ and } |\tilde{\xi}|\geq \de  \}\\
	& \mf E_{\de}^{-} = \{(\xi,\tilde{\xi}) \in \HH~:~   (\xi,\xi\tilde{\xi}^{-1}) \in \mf E_0 \text{ and } |\tilde{\xi}|\geq \de  \}
		\end{align*}
		\textbf{Claim:} $(\xi,\eta)\mapsto u(\xi)(v(\xi)-v(\eta)) \mc K(\eta^{-1}\xi ) \in L^1(\mf E_{\de}d\xi d\eta)$. \\
	For $\de$ small and $|\eta^{-1}\xi|\geq \de$, we have  $\theta (\eta^{-1}\xi) \geq \de^2$, and using properties of $\mc K$ and the fact that $ v=0$ in $\Om^c$, we deduce that
\begin{align*}
\int_{\mf E_{\de}} |u(\xi)|& |v(\xi)-v(\eta)| \mc K(\eta^{-1}\xi )d \xi  d \eta\\
 & = \ds \int_{\{\Om \times \H \}\cap \{|\eta^{-1}\xi|\geq \de\}} |u(\xi)||v(\xi)-v(\eta)| \mc K(\eta^{-1}\xi )d \xi  d \eta\\
& \qquad + \int_{\{\Om \times \Om^c \}\cap \{|\eta^{-1}\xi|\geq \de\}} |u(\eta)||v(\xi)-v(\eta)| \mc K(\eta^{-1}\xi )d \xi  d \eta\\
&\leq  2 \|v\|_{L^\infty(\H)} \left(\ds \int_{\{\Om \times \H \}\cap \{|\eta^{-1}\xi|\geq \de\}} |u(\xi)||v(\xi)-v(\eta)| \mc K(\eta^{-1}\xi )d \xi  d \eta \right.\\
&\left.\hspace{4cm} + \int_{\{\Om \times \Om^c \}\cap \{|\eta^{-1}\xi|\geq \de\}} |u(\eta)||v(\xi)-v(\eta)| \mc K(\eta^{-1}\xi )d \xi  d \eta
\right)\\
& \leq \frac{2 \|v\|_{L^\infty(\H)}}{\de^2} \left(\ds \int_{\{\Om \times \H \}\cap \{|\eta^{-1}\xi|\geq \de\}} |u(\xi)|\theta (\eta^{-1}\xi) \mc K(\eta^{-1}\xi )d \xi  d \eta \right.\\
&\left.\hspace{4cm} + \int_{\{\Om \times \Om^c \}\cap \{|\eta^{-1}\xi|\geq \de\}} |u(\eta)|\theta (\eta^{-1}\xi) \mc K(\eta^{-1}\xi )d \xi  d \eta
\right)\\
&\leq \frac{2 \|v\|_{L^\infty(\H)}}{\de^2} \left(\ds \int_{\{\Om \times \H \}} |u(\xi)|\theta (\tilde{\xi}) \mc K(\tilde{\xi})d \xi  d \tilde{\xi}  + \|u\|_{L^\infty(\Om^c)} \int_{\{\Om \times \Om^c \}} \theta (\tilde{\xi}) \mc K(\tilde{\xi} )d \xi  d \tilde{\xi}
\right)\\
&\leq  \frac{2 \|v\|_{L^\infty(\H)}}{\de^2}\left(\|u\|_{L^1(\Om)}+|\Om|\|u\|_{L^\infty(\Om^c)} \right)\ds \int_{\H }\theta (\tilde{\xi}) \mc K(\tilde{\xi})d \tilde{\xi}  <\infty.
\end{align*}
Hence the claim. Similarly, $(\xi,\eta)\mapsto u(\eta)(v(\xi)-v(\eta)) \mc K(\eta^{-1}\xi ) \in L^1(\mf E_{\de})$. It implies that
\begin{equation}\label{hf20}
\begin{aligned}
& \int_{\mf E_{\de}} (u(\xi)-u(\eta))(v(\xi)-v(\eta)) \mc K(\eta^{-1}\xi )d \xi  d \eta \\
& = \int_{\mf E_{\de}} u(\xi)(v(\xi)-v(\eta)) \mc K(\eta^{-1}\xi )d \xi  d \eta+ \int_{\mf E_{\de}} u(\xi)(v(\xi)-v(\eta)) \mc K(\xi^{-1}\eta )d \xi  d \eta\\
& = \int_{\mf E_{\de}^+} u(\xi)(v(\xi)-v(\xi\tilde{\xi})) \mc K(\tilde{\xi}^{-1})d \xi  d \tilde{\xi}+  \int_{\mf E_{\de}^-} u(\xi)(v(\xi)-v(\xi\tilde{\xi}^{-1})) \mc K(\tilde{\xi}^{-1})d \xi  d \tilde{\xi}\\
& = \int_{\mf E_{\de}^+} u(\xi)(v(\xi)-v(\xi\tilde{\xi})) \mc K(\tilde{\xi})d \xi  d \tilde{\xi}+  \int_{\mf E_{\de}^-} u(\xi)(v(\xi)-v(\xi\tilde{\xi}^{-1})) \mc K(\tilde{\xi})d \xi  d \tilde{\xi}
\end{aligned}
\end{equation}
Notice that if $(\xi, \tilde{\xi}) \in \mf E_{\de}^-\setminus E_{\de}^+$  then $(\xi,\xi\tilde{\xi}) \in \Om^c\times\Om^c$. It implies that
\begin{align*}
\int_{\mf E_{\de}^{-}\setminus \mf E_{\de}^+} u(\xi)(v(\xi)-v(\xi\tilde{\xi})) \mc K(\tilde{\xi})d \xi  d \tilde{\xi}=0
\end{align*}
	Similarly, 	$\ds \int_{\mf E_{\de}^{+}\setminus \mf E_{\de}^-}  u(\xi)(v(\xi)-v(\xi\tilde{\xi}^{-1})) \mc K(\tilde{\xi})d \xi  d \tilde{\xi}=0$. Using this with \eqref{hf20}, we obtain
	\begin{equation}\label{hf20unu}
	\begin{aligned}
	& \int_{\mf E_{\de}} (u(\xi)-u(\eta))(v(\xi)-v(\eta)) \mc K(\eta^{-1}\xi )d \xi  d \eta \\
		& = \int_{\mf E_{\de}^+\cup \mf E_{\de}^-} u(\xi)(v(\xi)-v(\xi\tilde{\xi})) \mc K(\tilde{\xi})d \xi  d \tilde{\xi}+  \int_{ \mf E_{\de}^+\cup\mf E_{\de}^-} u(\xi)(v(\xi)-v(\xi\tilde{\xi}^{-1})) \mc K(\tilde{\xi})d \xi  d \tilde{\xi}\\
		& = \int_{\mf E_{\de}^+\cup \mf E_{\de}^-} u(\xi)(2v(\xi)-v(\xi\tilde{\xi})-v(\xi\tilde{\xi}^{-1})) \mc K(\tilde{\xi})d \xi  d \tilde{\xi}
	\end{aligned}
	\end{equation}
	We now show that $\mf V(\xi,\tilde{\xi}):= (2v(\xi)-v(\xi\tilde{\xi})-v(\xi\tilde{\xi}^{-1}))\mc K(\tilde{\xi})  \in L^1(\HH)$. \\
	Notice that  $|\mf V(\xi,\tilde{\xi})|\leq 4\|v\|_{L^\infty(\H)} \mc K(\tilde{\xi}) $ and, by Taylor expansion
	\begin{align*}
|\mf V(\xi,\tilde{\xi})|\leq \|D^2v\|_{L^\infty(\H)}|\tilde{\xi}|^2 \mc K(\tilde{\xi}).
	\end{align*}
	It implies that
	\begin{align*}
	|\mf V(\xi,\tilde{\xi})|\leq 4\|v\|_{C^2(\H)}\theta(\tilde{\xi})\mc K(\tilde{\xi}).
	\end{align*}
Now we first recall the triangle inequality which states that  there exists $\underline{c}<1$ such that
	\begin{align}\label{hf21}
	\underline{c} \big||\xi|-|\eta|\big|\leq |\xi\eta| \text{ for all } \xi,\eta \in \H.
	\end{align}
	Choose $R>1$ such that $\Om \subset B_{\underline{c}R} $. If $ \xi \in (B_{2R})^c, \xi\tilde{\xi} \in\Om $ and $\xi\tilde{\xi}^{-1} \in \Om \subset B_{\underline{c}R}$ then by \eqref{hf21}, we get
	\begin{align}\label{hf22}
	|\tilde{\xi}|\geq |\xi|-\frac{1}{\underline{c}}|\xi\tilde{\xi}| \geq 2R-\frac{1}{\underline{c}} \underline{c}R = R>1.
	\end{align}
	Define
	\begin{align*}
& \mf R = \left\{ (\xi,\tilde{\xi}) \in \HH~:~ \xi \in (B_{2R}(0))^c \text{ and }  \left(\tilde{\xi} \in B_{\underline{c}R}(\xi^{-1})\text{ or } \tilde{\xi}^{-1} \in B_{\underline{c}R}(\xi^{-1}) \right)   \right\}\\
& \mf R_* = \left\{ (\xi,\tilde{\xi}) \in \HH~:~ \xi^{-1} \in B_{\underline{c}R}(\tilde{\xi}^{-1}) \cup B_{\underline{c}R}(\tilde{\xi}) \text{ and }  \tilde{\xi} \in (B_{1}(0))^c  \right\}.
	\end{align*}
	Let $ (\xi,\tilde{\xi}) \in  \mf R$ then
	\begin{align}\label{hf23}
	(|\xi|>2R \text{ and } |\xi\tilde{\xi}|<\underline{c}R) \text{ or }   (|\xi|>2R \text{ and } |\xi\tilde{\xi}^{-1}|<\underline{c}R).
	\end{align}
	Taking into account  \eqref{hf22} and \eqref{hf23}, we deduce that either  $|\tilde{\xi}|>1$ and $ \xi^{-1} \in  B_{\underline{c}R}(\tilde{\xi}) $ or  $|\tilde{\xi}|>1$ and $ \xi^{-1} \in  B_{\underline{c}R}(\tilde{\xi}^{-1}) $. It implies that $\mf R\subset \mf R_*$. Furthermore,  if $ (\xi,\tilde{\xi}) \in  \mf R_*$ then
	\begin{align}\label{hf24}
	\mc K(\tilde{\xi})= \theta(\tilde{\xi})\mc K(\tilde{\xi})
	\end{align}
	Now let $(\xi,\tilde{\xi}) \in ((B_{2R}(0))^c\times\H )\setminus \mf R$. Then $\xi \in (B_{2R}(0))^c,~ \tilde{\xi} \in \H,~ (\xi,\tilde{\xi}) \not \in \mf R$. That is, $\xi \in (B_{2R}(0))^c,~ |\xi \tilde{\xi}|>\underline{c}R,$ and $|\xi \tilde{\xi}^{-1}|>\underline{c}R$. As a result
\begin{align}\label{hf25}
v(\xi)=v(\xi \tilde{\xi})= v(\xi \tilde{\xi}^{-1})=0 \text{ for all } (\xi,\tilde{\xi}) \in ((B_{2R}(0))^c\times\H )\setminus \mf R.
\end{align}
	Using \eqref{hf24}, \eqref{hf25}  and definition of $\mc K$, we get
	\begin{equation}\label{hf26}
	\begin{aligned}
	\int_{(B_{2R}(0))^c\times\H} |\mf V(\xi,\tilde{\xi})|&d\xi d\tilde{\xi}
	= 	\int_{\mf R_*} |2v(\xi)-v(\xi\tilde{\xi})-v(\xi\tilde{\xi}^{-1})| \mc K(\tilde{\xi})d\xi d\tilde{\xi}\\
	&  \leq 4\|v\|_{L^\infty(\H)}\int_{\mf R_*}  \mc K(\tilde{\xi})d\xi d\tilde{\xi}  = 4\|v\|_{L^\infty(\H)}\int_{\mf R_*}  \theta(\tilde{\xi})\mc K(\tilde{\xi})d\xi d\tilde{\xi}\\
	& \leq C(Q,R) \|v\|_{L^\infty(\H)}\int_{\H}  \theta(\tilde{\xi})\mc K(\tilde{\xi})d\xi d\tilde{\xi}= C(Q,R,\mc K)<\infty .
	\end{aligned}
	\end{equation}
	Consider
	\begin{equation}\label{hf26bis}
	\begin{aligned}
	\int_{\HH} |\mf V(\xi,\tilde{\xi})|&d\xi d\tilde{\xi}
	\leq \int_{B_{2R}\times \H} |\mf V(\xi,\tilde{\xi})|d\xi d\tilde{\xi}+  C(Q,R,\mc K)\\
	& \leq C(Q,R,\mc K)+ 4\|v\|_{C^2(\H)}\int_{B_{2R}\times \H} \theta(\tilde{\xi})\mc K(\tilde{\xi}) d\xi d\tilde{\xi}\\
	& =  C(Q,R,\mc K)+ 4\|v\|_{C^2(\H)}|B_{2R}|\int_{ \H} \theta(\tilde{\xi})\mc K(\tilde{\xi}) d\xi d\tilde{\xi}  <\infty .
	\end{aligned}
	\end{equation}
	Hence we show that $ \mf V \in L^1(\HH)$. Taking  into account  Lemma \ref{lemhf3} (i),  \eqref{hf26bis}, and passing limit $\de \ra 0 $  in \eqref{hf20unu}, we get
		\begin{equation}\label{hf20doi}
	\begin{aligned}
	 \int_{\mf E_{0}} (u(\xi)-u(\eta))& (v(\xi)-v(\eta)) \mc K(\eta^{-1}\xi )d \xi  d \eta \\
& = \int_{\mf E_{0}^+\cup \mf E_{0}^-}  u(\xi)(2v(\xi)-v(\xi\tilde{\xi})-v(\xi\tilde{\xi}^{-1})) \mc K(\tilde{\xi})d \xi  d \tilde{\xi}.
	\end{aligned}
	\end{equation}
Using the definition of $\mf E_0^+ $ and $\mf E_0^-$ and the fact that $v=0$ in $\H\setminus\Om$,  we get \eqref{hf27}. \QED	
\end{proof}

\begin{Lemma}\label{lemhf4}
	Let $\psi \in C^2_0(\Om)$ and $u_n \in \mc Z$ be a sequence converging uniformly to $u_* \in \mc Z$ as $n \ra \infty$ and $u_n-u_* \in \mc Z_0$. Then
\begin{equation}\label{hf27unu}
\begin{aligned}
\lim_{n\ra \infty} \int_{\mc S} (u_n(\xi)-u_n(\eta))& (v(\xi)-v(\eta)) \mc K(\eta^{-1}\xi )d \xi  d \eta\\
&  =  \int_{\mc S} (u_*(\xi)-u_*(\eta)) (v(\xi)-v(\eta)) \mc K(\eta^{-1}\xi )d \xi  d \eta
\end{aligned}
\end{equation}
\end{Lemma}
\begin{proof}
	Using Proposition \ref{prophf1} to $u_n-u_*$,   for any $v \in C^2_0(\Om)$,  we deduce that
	\begin{align*}
	& \int_{\mc S} (u_n(\xi)-u_n(\eta)) (v(\xi)-v(\eta)) \mc K(\eta^{-1}\xi )d \xi  d \eta-   \int_{\mc S} (u_*(\xi)-u_*(\eta)) (v(\xi)-v(\eta)) \mc K(\eta^{-1}\xi )d \xi  d \eta\\
	&=  \int_{\mc S} ((u_n-u_*)(\xi)-(u_n-u_*)(\eta)) (v(\xi)-v(\eta)) \mc K(\eta^{-1}\xi )d \xi  d \eta\\
	& = \int_{\mc S} (u_n-u_*)(\xi)(2v(\xi)-v(\xi \eta)- v(\xi\eta^{-1})) \mc K(\eta^{-1}\xi )d \xi  d \eta.
	\end{align*}
	Since $ u_n\ra u_*$ uniformly in $\H$, we get
		\begin{align*}
	\lim_{n\ra \infty}\bigg| \int_{\mc S} (u_n-u_*)(\xi)(2v(\xi)-v(\xi \eta)- v(\xi\eta^{-1})) \mc K(\eta^{-1}\xi )d \xi  d \eta\bigg|\\
	 \leq \lim \|u_n-u_*\|_{L^\infty(\Om)} \|\mf V\|_{L^1(\HH)}=0
	\end{align*}
	where $\mf V$ is defined in Proposition \ref{prophf1}. Hence \eqref{hf27unu} holds true. \QED
\end{proof}

\textbf{Proof of Theorem \ref{thmhf1}: }
 Let $u$ be a solution to variational inequality \eqref{hf28}. In the framework  of Theorem \ref{thmhf1}, we now apply Theorem  \ref{thmhf3} and deduce that \eqref{hf29} holds. \QED

\section{Proof of Theorem \ref{thmhf2}}
In this Section, We proove trhe existence of mountain pass solution of the problem. Notice that the problem $(\mc P)$ has a variational Structure and the  energy functional associated to the problem $(\mc P)$ is given as
\begin{align*}
\mc H(u)=\frac12\int_{\mc S} |u(\xi)-u(\eta)|^2\mc K(\eta^{-1}\xi)~d\xi d\eta-\int_{\Om} F(\xi,u(\xi))~dx.
\end{align*}
Observe that $\mc H \in C^1(\mc Z_0,\mathbb{R})$ and for $u \in \mc Z_0$ and for any $ v \in \mc Z_0$,
\begin{align*}
\ld \mc H^\prime(u), v \rd = \int_{\mc S} (u(\xi)-u(\eta)) (v(\xi)-v(\eta)) \mc K(\eta^{-1}\xi )~d \xi  d \eta -\int_{\Om} f(\xi,u(\xi))~d\xi
\end{align*}
\begin{Remark}\label{remhf1}
	Under the assumptions \eqref{hf30}, for any $\e>0$ there exists $\de >0$ such that a.e. $\xi \in \Om$ and for any $ l \in \mathbb{R}$,  we have
	\begin{align}\label{hf31}
	|f(\xi,l)| \leq 2\e|l| +q\de|l|^{q-1} \text{ and }  |F(\xi,l)| \leq \e|l|^2 +\de|l|^{q}.
	\end{align}
	 Also, there exist two positive functions $m$ and $M$ belonging in $L^\infty(\Om)$ such that a.e. $\xi \in \Om$ and for any $l \in \mathbb{R}$,
	 \begin{align}\label{hf32}
	 F(\xi,l)\geq m(\xi)|l|^{\vartheta}-M(\xi).
	 \end{align}
\end{Remark}

\begin{Proposition}\label{prophf3}
	The following holds
	\begin{enumerate}
		\item [(i)] There exists $\al,\rho>0$ such that $\mc H(u)\geq \al$ for  $\|u\|_{\mc Z_0}= \rho$.	
		\item [(ii)] There exists  $e \in \mc Z_0$  such that $e\geq 0$ a.e. in $\H,~\|e\|>\rho$ and  $\mc H(e)<\beta$.
	\end{enumerate}
\end{Proposition}
\begin{proof}
	From  H\"older's inequality, Lemma \ref{lemhf1}, and  \eqref{hf31}, we deduce that
	\begin{equation*}
	\begin{aligned}
	\mc H(u)& \geq \frac12 \|u\|^2_{\mc Z_0}-\e\int_{\Om} |u|^2~d\xi-\de\int_{\Om} |u|^q~d\xi\\
	& \geq  \frac12 \|u\|^2_{\mc Z_0} -\e |\Om|^{(Q^*-2)/Q^*} \|u\|_{L^{Q^*}(\Om)}^2-\de |\Om|^{(Q^*-q)/Q^*}\|u\|_{L^{Q^*}(\Om)}^q\\
	& \geq \left(\frac12 -\e |\Om|^{(Q^*-2)/Q^*}c(Q,s) c^2(\mu) \right) \|u\|^2_{\mc Z_0}- \de |\Om|^{(Q^*-q)/Q^*}(c(Q,s) c^2(\mu))^{q/2}\|u\|^q_{\mc Z_0}.
	\end{aligned}
	\end{equation*}
	Choose $\e$ such that $\e |\Om|^{(Q^*-2)/Q^*}c(Q,s) c^2(\mu)<\frac12 $. It implies that
		\begin{equation*}
	\begin{aligned}
	\mc H(u)& \geq \ba  \|u\|^2_{\mc Z_0} -\kappa \|u\|^{q}_{\mc Z_0}.
	\end{aligned}
	\end{equation*}
	Since $q>2$, so we can choose $\al,\rho>0$ such that $\J(u)\geq \al$ for  $\|u\|_{\mc Z_0}= \rho$. \\
(ii)  Let $u \in \mc Z_0$ and $ j>0$.  From \ref{hf32}, we obtain
\begin{align*}
\mc H(ju)& = \frac{j^2}{2} \|u\|_{\mc Z_0}^2 - \int_{\Om} F(\xi,ju(\xi))~dx \\
& \leq  \frac{j^2}{2} \|u\|_{\mc Z_0}^2 - j^{\vartheta}\int_{\Om}  m(\xi)|u(\xi)|^{\vartheta} ~d\xi+ \int_{\Om} M(\xi)~d\xi\ra -\infty
\end{align*}
as $j \ra \infty$. Let $e= ju$ for $j$ large  then (ii) follows. \QED
\end{proof}

\begin{Proposition}\label{prophf4}
	Let $u_n$ be a sequence in $\mc Z_0$ such that  $\mc H(u_n) \ra c $ and $\|\mc H^\prime(u_n)\|\ra 0$ as $n \ra \infty$. Then there exists $u_* \in X_0$ such that, up to subsequence, $u_n \ra u_*$ in $\mc Z_0$.
\end{Proposition}

\begin{proof}
	First we prove that $u_n$ is a bounded sequence in $\mc Z_0$. Using  \eqref{hf31} with $\e=1$, we get
	\begin{align}\label{hf33}
\bigg| \int_{\Om\cap \{|u_n|\leq \mc R \} }  \left( F(\xi,u_n(\xi))-\frac{1}{\vartheta} f(\xi,u_n(\xi))u_n(\xi)\right)d\xi \bigg| \leq \left(\mc R^2+\de \mc R^q+\frac{2}{\vartheta}\mc R+\frac{q}{\vartheta}\de \mc R^{q-1} \right) |\Om|.
	\end{align}
	Furthermore from the assumptions on the sequence $\mc Z_0$, we can choose $\tilde{c}>0$ such that
	\begin{align}\label{hf34}
|\mc H(u_n)| \leq \tilde{c} \text{ and } \bigg| \bigg\ld \mc H^\prime(u_n), \frac{u_n}{\|u_n\|_{\mc Z_0}}\bigg\rd \bigg|\leq \tilde{c}
	\end{align}
	Taking into account \eqref{hf30}, \eqref{hf33}, and \eqref{hf34}, we conclude
	\begin{align*}
	\tilde{c}(1+\|u\|_{\mc Z_0})& \geq \mc H(u_n)-\frac{1}{\vartheta} \ld  \mc H^\prime(u_n),u_n \rd \\
	& \geq \left(\frac12-\frac{1}{\vartheta}\right)\|u\|^2_{\mc Z_0}-  \int_{\Om\cap \{|u_n|\leq \mc R \} }  \left( F(\xi,u_n(\xi))-\frac{1}{\vartheta} f(\xi,u_n(\xi))u_n(\xi)\right)d\xi \\
	& \geq \left(\frac12-\frac{1}{\vartheta}\right)\|u\|^2_{\mc Z_0}-c_1.
	\end{align*}
	It implies that $u_n$ is a bounded sequence in $\mc Z_0$. Since $\mc Z_0$ is a Hilbert space. So up to a subsequence there exists $u_* \in \mc Z_0$ such that $u_n \rp u_*$  weakly in $\mc Z_0$. From Lemma \ref{lemhf2}, we have
	\begin{equation}\label{hf35}
	\begin{aligned}
& 	u_n \ra u_*  \text{ in } L^q(\H);~ u_n \ra u_*\text{  a.e  in  } \H;\\
	& \text{  there exists } w\in L^q(\H) \text{  such that  } |u_n(\xi)|\leq w(\xi) \text{  a.e  in  } \H \text{ and } n \in \mathbb{N}
	\end{aligned}
	\end{equation}
Taking into account \eqref{hf30}, \eqref{hf35}, and  dominated convergence theorem, we get
\begin{equation}\label{hf36}
\begin{aligned}
& \int_{\Om} f(\xi,u_n(\xi))u_n(\xi)~d\xi\ra \int_{\Om} f(\xi,u_*(\xi))u_*(\xi)~d\xi;\\
&  \int_{\Om} f(\xi,u_n(\xi))u_*(\xi)~d\xi\ra \int_{\Om} f(\xi,u_*(\xi))u_*(\xi)~d\xi
	\end{aligned}
\end{equation}
Taking into account the fact that $\|\mc H^\prime(u_n)\|\ra 0$ as $n \ra \infty$ and \eqref{hf36}, we deduce that
\begin{equation}\label{hf37}
\begin{aligned}
& \|u_n\|^2_{\mc Z_0}\ra \int_{\Om} f(\xi,u_*(\xi))u_*(\xi)~d\xi;\\
&   \ld u_n,u_*\rd_{\mc Z_0} - \int_{\Om} f(\xi,u_n(\xi))u_*(\xi)~d\xi = \ld \mc H^\prime(u_n),u_*\rd \ra 0 \text{ as } n \ra \infty.
\end{aligned}
\end{equation}
From \eqref{hf35}, \eqref{hf36}, and \eqref{hf37}, we get
\begin{equation*}
\begin{aligned}
\|u_n\|^2_{\mc Z_0}\ra \|u_n\|^2_{\mc Z_0}\text{ as } n \ra \infty.
\end{aligned}
\end{equation*}
Hence $\|u_n-u_*\|^2_{\mc Z_0}= \|u_n\|^2_{\mc Z_0}+\|u_*\|^2_{\mc Z_0} -2 \ld u_n,u_*\rd_{\mc Z_0} \ra 0  \text{ as } n \ra \infty. $
This completes the proof. \QED
\end{proof}
\textbf{Proof of Theorem \ref{thmhf2}: }  Using Mountain Pass Theorem along with Propositions \ref{prophf3} and \ref{prophf4}, there exists a critical point $u_0 \in \mc Z_0$. Also, $\mc H(u_0) \geq \al >0= \mc H(0)$. It implies $u \not \equiv 0$. \QED

\end{document}